\documentclass[11pt,A4]{article}
\usepackage{amsmath,amsthm,amsfonts,amssymb,amscd, amsxtra,mathrsfs, color}
\usepackage[english]{babel}
\usepackage{setspace}

\usepackage[margin=1.0in]{geometry}
\newtheorem{theorem}{Theorem}
\newtheorem{lemma}[theorem]{Lemma}

\newtheorem{remark}[theorem]{Remark}
\newtheorem{definition}[theorem]{Definition}

\newtheorem{algorithm}{Algorithm}

\begin{document}

\title{A Global Version of the Newton Method for Finding a Singularity of the Nonsmooth Vector Fields on Riemannian Manifolds}

\author{Fabiana R. de Oliveira \thanks{Instituto de Matem\'atica e Estat\'istica, Universidade Federal de Goi\'as,  CEP 74001-970 - Goi\^ania, GO, Brazil, E-mails: {\tt  rodriguesfabiana@ufg.br.com}, {\tt  fabriciarodrigues@ufg.br}. }
\and
Fabr\'icia R. Oliveira \footnotemark[1]
}

\maketitle
\begin{abstract}
This paper is concerned with an algorithm for finding a singularity of the nonsmooth vector fields. Firstly, we discuss the main results of the Newton method presented in \cite{deOliveiraFerreira2020} for solving the aforementioned problem. Combining this method with a nonmonotone line search strategy, we then propose a global version of the Newton Method. Finally, numerical experiments illustrating the practical advantages of the proposed scheme are reported.
\end{abstract}



\section{Introduction}\label{sec:int}
In this paper, we consider the problem of finding a singularity of the nonsmooth vector fields defined on Riemannian manifolds, by means of a global version of the Newton method. Although the interest in nonsmooth functions in the Riemannian setting has increased; see for example \cite{Azagra2005,Ferreira2008,HosseiniEtall2017,GrohsHossein2016,HosseiniHuangYousefpour2018,Hosseini2013,Hosseini2017,LedyaevYuZhu2007}, only a few studies exist on nonsmooth vector fields in this context; see \cite{HosseiniEtall2017,Rampazzo2007}. Recently, \cite{deOliveiraFerreira2020} proposed and analyzed a version of the Newton method for finding a singularity of a class of locally Lipschitz continuous vector field. For the smooth vector fields, much has already been done, see \cite{ArgyrosHilout2009,BittencourtFerreira2015,FernandesFerreiraYuan2017,OrizonSilva2012,FerreiraSvaiter2002,ChongWang2005,Wang2011}. In, \cite{BortoliFernandesFerreiraYuan2020} was proposed a global version of the Newton method, called damped Newton method, for finding a singularity of smooth vector fields. In particular, the basic idea of this method is to use a linear search when the full step does not provide a sufficient decrease for values of the chosen merit function. Owing to the aforementioned facts, we believe that the development of news schemes for nonsmooth vector fields might be of significant interest.

Newton method is very popular by their fast local convergence, however, it is very sensitive with respect to the initial iterate and may diverge if it is not sufficiently close to the solution. To bypass this drawback, some strategies have been incorporated on the Newton method, for example, BFGS, Levenberg--Marquardt and Trust Region; see \cite{BehlingFischerHerrich2014,Bertsekas2014,Dennis1996,Fan2013,NoceWrig06}. Another strategy of particular interest is the one by using a nonmonotone linear search together with a merit function, see \cite{Grippo1986,IannazzoPorcelli2018}. It is worth pointing out that the nonmonotone strategies have been shown more efficient than monotone ones owing to the fact that enforcing the monotonicity of the function values may make the method to converge slower.

The goal of this paper is to present a global version of the Newton method for finding a singularity of nonsmooth vector fields. Basically, we combine our first algorithm presenting in \cite{deOliveiraFerreira2020} with the nonmonotone line search strategy. We show that any accumulation point of the iterative sequence is a stationary point of the chosen merit function. To the convergence rate, we ensure that the sequence generated by the proposed method reduces to a sequence generated by the Newton method after a finite number of iterations. Moreover, to assess the practical behavior of the new scheme, some numerical experiments are reported. In particular, we present a scenario in which the global version becomes interesting in practice.

This paper is organized as follows. In Section~\ref{sec:int.1}, some notations and basic results are presented. In Section~\ref{sec:na}, we discuss some main results of  nonsmooth analysis and of the Newton method to the Riemannian context. In Section~\ref{sec:nmca}, we describe a global version of the Newton method and establish its convergence theorems. In Section~\ref{sec:numexp}, we present some numerical experiments of the proposed scheme. Finally, some concluding remarks are given in Section~\ref{sec:fr}.

\section{Notations and Definition} \label{sec:int.1}
In this section, we recall some notations, definitions, and basic properties used herein, see, for example, \cite{Lang1995,Sakai1996,Loring2011}. Let $\mathbb{M}$ be an $n$-dimensional smooth Riemannian manifold with {\it Riemannian metric} denoted by  $\langle  \cdot,   \cdot \rangle$ and the corresponding {\it norm}  by $\|\cdot\|$.  The {\it length} of a piecewise smooth curve  $\gamma:[a,b]\rightarrow \mathbb{M}$ joining $p$ to $q$ in $\mathbb{M}$, i.e., $\gamma(a)= p$ and $\gamma(b)=q$, is denoted by  $\ell(\gamma)$.  The {\it Riemannian distance} between $p$ and $q$ is defined as $d(p,q) = \inf_{\gamma \in \Gamma_{p,q}} \ell(\gamma)$,  where $\Gamma_{p,q}$ denotes the set of all the piecewise smooth curves in $\mathbb{M}$ joining points $p$ and $q$. This distance induces the original topology on $\mathbb{M}$; i.e., $(\mathbb{M}, d)$ is a complete metric space and the bounded and closed subsets are compact. The {\it open ball} of radius $\delta > 0$, centered at $p$ is defined by $B_{\delta}(p):= \{ q\in \mathbb{M}:~ d(p,q) < \delta\}$. The {\it tangent space} at point $p$ is denoted by $T_p\mathbb{M}$,   the {\it tangent bundle}  by $T\mathbb{M} := \bigcup_{p\in \mathbb{M}}T_p\mathbb{M}$,  and a {\it vector field} by a mapping   $X: \mathbb{M} \to T\mathbb{M}$ such that $X(p) \in T_p\mathbb{M}$.  Let $\gamma$ be a curve joining points $p$ and $q$ in $\mathbb{M}$, and let $\nabla$ be the Levi--Civita connection associated to $(\mathbb{M}, \langle \cdot, \cdot \rangle)$. For each $t \in [a,b]$, $\nabla$ induces a linear isometry between the tangent spaces $T _{\gamma(a)} {\mathbb{M}}$ and $T _{\gamma(t)} {\mathbb{M}}$, relative to $\langle \cdot , \cdot \rangle$, defined by $P_{\gamma,a,t}v = Y(t)$, where $Y$ denotes the unique vector field on $\gamma$ such that $\nabla_{\gamma'(t)}Y(t) = 0$ and $Y(a)= v$. The aforementioned isometry is called {\it parallel transport} along the  segment $\gamma$ joining  $\gamma(a)$ to $\gamma(t)$. It can  be showed   that
$P_{\gamma,\,b,\,t}\circ P_{\gamma,\,a,\,b}= P_{\gamma,\,a,\,t}$ and $P_{\gamma,\,t,\,a}=P^{-1}_{\gamma,\,a,\,t}, $ for all $a\leq b\leq t$.
For simplicity and convenience, whenever there is no confusion, we consider the notation $P_{\gamma,p,q}$  instead of $P_{\gamma,\,a,\,b}$,  where $\gamma$ denotes a segment joining $p$ to $q$,  with $\gamma(a)=p$ and $\gamma(b)=q$.  We use the short  notation $P_{pq}$  instead of  $P_{\gamma,p,q}$ whenever  there exists  a  unique geodesic   segment joining $p$ to $q$.  Let $f: \mathbb{M} \rightarrow \mathbb{R}$ be a smooth function, the \textit{Riemannian gradient} $f'(p)$ of $f$ at $p\in \mathbb{M}$ is defined as the unique element in $T_p\mathbb{M}$ such that
\begin{equation}\label{eq:DefGra}
f'(p)^T\xi_p = Df(p)[\xi_p], \qquad \forall ~ \xi_p \in T_p\mathbb{M},
\end{equation}
where $Df(p): T_p\mathbb{M} \rightarrow T_p\mathbb{M}$ is the differential of $f$ at $p$. A vector field $Y$ along the smooth curve $\gamma$ in $\mathbb{M}$ is {\it parallel}  when $\nabla_{\gamma^{\prime}} Y=0$. If $\gamma^{\prime}$ is parallel, we say that $\gamma$ is a {\it geodesic}. Because the geodesic equation $\nabla_{\gamma'}\gamma' = 0$ is a second-order nonlinear ordinary differential equation, the geodesic $\gamma$ is determined using its position $p$ and velocity $v$ at $p$. It is easy to check that $\left\| \gamma' \right\|$ is constant. The restriction of a geodesic to a  closed bounded interval is called a {\it geodesic segment}. A geodesic segment  joining $p$ to $q$ in $\mathbb{M}$ is {\it minimal} if its length is equal to $d(p,q)$, and, in this case, it will be denoted by $\gamma_{pq}$. A Riemannian manifold is {\it complete} if its geodesics $\gamma(t)$ are defined for any value of $t\in \mathbb{R}$. The Hopf--Rinow theorem asserts that any pair of points in a complete Riemannian  manifold $\mathbb{M}$ can be joined by a (not necessarily unique) minimal geodesic segment. Hereinafter, {\it $\mathbb{M}$ denotes  an $n$-dimensional smooth and complete Riemannian manifold}. Because of the completeness of the Riemannian manifold $\mathbb{M}$, the {\it exponential map} at $p$, $\exp_{p}:T_{p}\mathbb{M} \to \mathbb{M} $ can be given by $\exp_{p}v = \gamma(1)$, where $\gamma$ denotes the geodesic defined by its position $p$ and velocity $v$ at $p$, and $\gamma(t) = \exp_p(tv)$ for any value of $t$. The inverse of the exponential map (if exists) is denoted by $\exp^{-1}_{p}$. Let $p\in \mathbb{M}$, the {\it injectivity radius} of $\mathbb{M}$ at $p$ is defined by $ r_{p}:=\sup\{r > 0:~{\exp_{p}}_{\lvert_{B_{r}(0_{p})}} \mbox{ is\, a\, diffeomorphism} \}$, where $0_{p}$ denotes the origin of the $T_{p}\mathbb{M}$, and $B_{r}(0_{p}):= \{v\in T_{p}\mathbb{M}:~\| v-0_{p}\| <r\}$.
\begin{remark}
For $\bar{p}\in \mathbb{M}$, the above definition implies that if $0<\delta<r_{\bar{p}}$, then $\exp_{\bar{p}}B_{\delta}(0_{ \bar{p}}) = B_{\delta}( \bar{p})$. Therefore, for all $p, q\in B_{\delta}(\bar{p})$, there exists a unique geodesic segment $\gamma$ joining  $p$ to $q$, given by $\gamma_{p q}(t)=\exp_{p}(t \exp^{-1}_{p} {q})$ for all $t\in [0, 1]$ and $d(p,q) = \|\exp_p^{-1}q\|$.
\end{remark}

Next, we present a quantity that plays  an important role in the sequel;  it  was defined in \cite{Dedieu2003}.
\begin{definition} \label{def:kp}
Let $p \in \mathbb{M}$ and $r_{p}$ be the radius of injectivity of $\mathbb{M}$ at $p$. We define the quantity as follows:
$$
K_{p}:=\sup\left \{\dfrac{d(\exp_{q}u, \exp_{q}v)}{\left\| u-v\right\|} :~ q\in B_{r_{p}}(p), ~ u,\,v\in T_{q}\mathbb{M}, ~u\neq v, ~\| v\| \leq r_{p},
\| u-v\|\leq r_{p}\right\}.
$$
\end{definition}

In the following remark, we show that  an estimative for the value of  $K_p$ can be found for Riemannian manifolds with  non-negative sectional curvature.
\begin{remark}
The number $K_{p}$ measures how fast the geodesics spread apart in $\mathbb{M}$. Particularly, when $u =~0$ or, more generally, when $u$ and $v$ are on the same line through $0$, then $d(\exp_{q}u, \exp_{q}v)=\| u-v\|$. Therefore, $K_{p}\geq1$ for all $p\in \mathbb{M}$. When $\mathbb{M}$ has non-negative sectional curvature, the geodesics spread apart less than the rays \cite[Chapter 5]{doCarmo1992}, i.e., $d(\exp_{p}u, \exp_{p}v)\leq\|u-v\|$; in this case, $K_{p}=1$  for all $p\in \mathbb{M}$.
\end{remark}
\begin{definition}
The {\it directional derivative} of   $X$ at $p$ along the direction $v\in T_p\mathbb{M}$ is defined by
\begin{align*}\label{derdir}
\nabla X(p, v) := \lim_{t \downarrow 0}\dfrac{1}{t}\left[P_{\exp_{p}(tv) p}X(\exp_{p}(tv)) - X(p)\right] \in T_p\mathbb{M},
\end{align*}
whenever  the limit exists, where $P_{\exp_{p}(tv) p}$  denotes the parallel transport along $\gamma(t)= \exp_{p}(tv)$.
\end{definition}

In particular, if the directional derivative exists for every $v$, the vector field $X$ is \textit{directionally differentiable} at $p$. We end this section with two  definitions know, namely norm of a linear mapping and descent direction for functions on Riemannian manifolds.
\begin{definition}\label{de:normmult}
Let $p \in \mathbb{M}$. The  norm of a linear mapping $A: T_p\mathbb{M} \to T_p\mathbb{M}$  is defined by
$$
\|A\|:=\sup \left\{ \|A v  \|:~  v \in T_p\mathbb{M}, ~\|v\| = 1 \right\}.
$$
\end{definition}
\begin{definition}
Let $f: \mathbb{M} \rightarrow \mathbb{R}$ be a continuously differentiable function in a neighborhood of $p \in \mathbb{M}$. A vector $v \in T_p\mathbb{M}$, is called a descent direction for $f$ at $p$ if satisfies $f'(p)^{T} v < 0$, where $f'(p)^{T}$ is the transpose of the gradient of $f$ at $p$.
\end{definition}

\section{Preliminary Results} \label{sec:na}
Here, we discuss on Riemannian settings the main results of nonsmooth analysis studied in \cite{deOliveiraFerreira2020}. We begin by presenting the concept of locally Lipschitz continuous vector fields. This concept was introduced in \cite{CruzNetoLimaOliveira1998} for gradient vector fields, and its extension to general vector fields can be found in \cite[p.~241]{CanaryEpsteinMarden2006}.
\begin{definition} \label{def:ILC}
A vector field $X:\Omega\subseteq \mathbb{M} \rightarrow T_p\mathbb{M}$ is regarded as locally Lipschitz continuous if for each $\bar{p} \in \Omega$ there exist constants $L, \delta > 0$, such that $\|P_{\gamma, p, q}X(p) - X(q)\| \leq  L \,\ell(\gamma)$ for all $p, q  \in B_{\delta}(\bar{p})$ and all geodesic segment $\gamma$ joining $p$ to $q$.
\end{definition}

\begin{remark}
According to the Rademacher theorem, see \cite[Theorem 3.2]{deOliveiraFerreira2020}, locally Lipschitz continuous vector fields are everywhere differentiable.
\end{remark}

Next, we define the {\it Clarke--generalized covariant derivative} of a vector field, which has appeared in \cite{deOliveiraFerreira2020}. This derivative requires only the local Lipschitz continuity of the vector field $X$ and its well-definedness is ensured by Rademacher theorem.
\begin{definition}\label{def:general}
The Clarke--generalized covariant derivative of a locally Lipschitz continuous vector field $X$ is a set-valued mapping $\partial X: \mathbb{M}  \rightrightarrows T\mathbb{M}$ defined as
$$
\partial X(p) :=  \mbox{conv} \left\{ H\in {\mathcal{L}}(T_p\mathbb{M}):~ \exists\, \{p_k\}\subset {\cal D}_X,~ \lim_{k\to \infty}p_k =p,\, H = \lim_{k\rightarrow \infty}P_{p_kp}\nabla X(p_k) \right\},
$$
where $\mbox{``conv"}$ represents the convex hull, ${\mathcal{L}}(T_p\mathbb{M})$ the vector space that comprises all the linear operators from $T_p\mathbb{M}$ to $T_p\mathbb{M}$, and ${\cal D}_X$ the set of points at which $X$ is differentiable.
\end{definition}
\begin{remark}
According Definition \ref{def:general} and \cite[Corollary 3.1]{FernandesFerreiraYuan2017}, it is evident that if $X$ is differentiable near $p$, and if its covariant derivative is continuous at $p$, then $\partial X(p) = \{\nabla X(p)\}$. Otherwise, $\partial X(p)$ could contain other elements that are different from $\nabla X(p)$, even if $X$ is differentiable at $p$ (see \cite[Example 2.2.3]{Clarke1990}). In \cite[Proposition 3.1]{deOliveiraFerreira2020} were established important results for the Clarke--generalized covariant derivative. The results established there that will be useful in our study are: $(i)$ no vacuity of the set $\partial X(p)$ for all $p\in \mathbb{M}$ and $(ii)$ local limitation for set-valued mapping $\partial X: \mathbb{M}\rightrightarrows T\mathbb{M}$, i.e., for all $\delta > 0$ and any $p\in \mathbb{M}$, there exists a $L > 0$, such that $\|V\|\leq L$ for all $q \in B_{\delta}(p)$ and all $V \in \partial X(q)$.
\end{remark}

The next definition is important for the discussions hereinafter, as it presents an important condition for ensuring the well-definedness of the sequence generated by the Newton methods even when there is no differentiability.
\begin{definition}\label{def:reg}
We say that a vector field $X$  on $\mathbb{M}$ is regular at $p \in \mathbb{M}$ if all $V_p \in \partial X(p)$ are non-singular.  If $X$ is regular at every point of $\Omega \subseteq \mathbb{M}$, we say that $X$ is regular on $\Omega$.
\end{definition}

Next, we establish for the locally Lipschitz  continuous vector fields that, if $V_{\bar{p}}$ is non-singular there exists a neighborhood of $\bar{p} \in \mathbb{M}$ where $V_p$ is non-singular. The proof is analogous to \cite[Lemma 4.2]{deOliveiraFerreira2020}.
\begin{lemma}\label{le:NonSing}
Let $X$ be a locally Lipschitz  continuous vector field on $\mathbb{M}$. Assume that $X$  is   regular  at $\bar{p} \in \mathbb{M}$ and let $\lambda_{\bar{p}}\geq \max\{\|{V_{\bar{p}}^{-1}}\|: ~ {V_{\bar{p}}}\in \partial X({\bar{p}})\}$. Then, for every    $\epsilon>0$  satisfying  $\epsilon \lambda_{\bar{p}} <1$, there exists $0 < \delta < r_{\bar{p}}$  such that $X$ is regular on $ B_{\delta}(\bar{p})$ and
\begin{equation*} \label{eq:BanachLem}
\|V_p^{-1}\| \leq  \frac{\lambda_{\bar{p}}}{1 -\epsilon \lambda_{\bar{p}}}, \qquad  \forall~p\in   B_{\delta}({\bar{p}}), \quad \forall ~V_p\in \partial X(p).
\end{equation*}
\end{lemma}

As already mentioned, in this paper, we propose and investigate a global version of the Newton method for finding a singularity of a vector field $X$ on $\mathbb{M}$, i.e., to solve the following problem
\begin{equation} \label{eq:TheProblem}
\mbox {find}  \quad p\in \mathbb{M}\quad  \mbox{such that} \quad X(p)=0,
\end{equation}
where  $X$ denotes  a locally Lipschitz  continuous vector field on $\mathbb{M}$. In \cite{deOliveiraFerreira2020} was propose a version of Newton method (NM) for solving the problem~\eqref{eq:TheProblem}. The algorithm is described formally in the sequence.\\
\hrule
\begin{algorithm}  \label{Alg:NNM}
{\vspace{0.2cm}\bf Newton method \vspace{0.3cm}}
\hrule
\vspace{0.1cm}
\begin{description}
\vspace{.5 cm}
\item[\bf Step 0.] Let $p_0\in \mathbb{M}$ be given, and set $k=0$.
\item[\bf Step 1.] If $X(p_k) = 0$, then \textbf{stop}.
\item[\bf Step 2.] Choose a $V_k := V_{p_k}\in \partial X(p_k)$ and compute $p_{k+1}=\exp_{p_{k}}(-V_{k}^{-1}X(p_{k}))$.
\item[\bf Step 3.] Set $k\gets k+1$, and go to \textbf{Step~1}.
\vspace{.5 cm}
\end{description}
\hrule
\end{algorithm}
\vspace{0.3cm}
\noindent

The local convergence analysis of the Newton method described by Algorithm~\ref{Alg:NNM} was made under the following assumptions for the locally Lipschitz continuous vector field $X$.
\begin{itemize}
\item[{\bf A1.}]  Let  ${\bar p}\in \mathbb{M}$, $0<\delta<  r_{\bar p}$ and $X$ be regular on $B_{\delta}({\bar p})$. Consider $\lambda_{\bar p}\geq \max\{\|{V_{\bar p}^{-1}}\|: ~ { V_{\bar p}}\in \partial X({\bar p})\}$  and   $\epsilon >0$ satisfy $\epsilon \lambda_{\bar p}<1$. For all $p \in   B_{\delta}({\bar p})$ and all $V_{p}\in \partial X(p)$ there hold
\begin{eqnarray}
\displaystyle \|V_{p}^{-1}\| &\leq&  \frac{\lambda_{\bar p}}{1 -\epsilon \lambda_{\bar p}},   \label{eq:fcA1}    \\
\displaystyle \left\|X({\bar p})-P_{p{\bar p}}\left[X(p)+ V_{p}\exp^{-1}_{p} {\bar p}\right]\right\|&\leq& \epsilon \,d(p, {\bar p})^{1+\mu}, \qquad 0\leq \mu \leq 1. \label{eq:scA1}
\end{eqnarray}
\end{itemize}

It is worth mentioning that the semismooth and $\mu$-order semismooth vector fields for $0<\mu \leq 1$  satisfy inequalities~\eqref{eq:fcA1} and \eqref{eq:scA1}, see \cite{deOliveiraFerreira2020}.
\begin{definition}
Let  $0< \delta< r_{\bar p}$ be given by above assumption. The {\it Newton iteration  mapping} $N_{X} \colon B_{\delta}({\bar p}) \rightrightarrows \mathbb{M}$ for $X$ is defined by $N_{X} (p):= \{\exp_{p}(-V^{-1}_pX(p)): ~ V_p \in \partial X(p)\}$.
\end{definition}

In the following, we present a result about the behavior of the Newton iteration  mapping near a singularity of the vector field $X$, whose proof can be found in \cite[Lemma 4.1]{deOliveiraFerreira2020}.
\begin{lemma}\label{le:welldefined}
Suppose  that $p_* \in \mathbb{M}$ is  a solution of problem~\eqref{eq:TheProblem}, $X$  satisfies {\bf A1}  with ${\bar p} = p_*$ and the constants  $\epsilon>0$, $0<\delta<  r_{p_*}$ and $0\leq \mu \leq 1$  satisfy  $\epsilon \lambda_{p_*}(1+ \delta^{\mu}K_{p_*}) <1$.  Then, there exists   ${\hat \delta}> 0$ such that $X$ is regular on $ B_{{\hat \delta}}(p_*)$ and
\begin{equation*} \label{eq;csl}
d\left(\exp_{p}(-V_{p}^{-1}X(p)),p_{*}\right)\leq \frac{ \epsilon\lambda_{p_*}K_{p_*}}{1 - \epsilon \lambda_{p_*} }  d(p,p_{*})^{1+\mu},   \qquad \forall~p \in B_{\hat \delta}({p_*}),  \quad \forall ~V_{p}\in \partial X(p).
\end{equation*}
Consequently,  $N_{X}$ is well-defined on $B_{{\hat \delta}}(p_{*})$ and   $N_{X}(p) \subset   B_{{\hat \delta}}(p_{*})$   for all $p\in B_{\hat \delta}(p_{*})$.
\end{lemma}

We end this section with a result establishing the convergence rate for a sequence generated by Algorithm~\ref{Alg:NNM}. Its proof is a direct application of Lemma~\ref{le:welldefined}, see \cite[Theorem 4.1]{deOliveiraFerreira2020}.
\begin{theorem}\label{th:conv}
Suppose   that  $p_* \in \mathbb{M}$ is  a solution of problem~\eqref{eq:TheProblem}, $X$  satisfies {\bf A1}  with ${\bar p} = p_*$,  and the constants  $\epsilon>0$, $0<\delta<  r_{p_*}$ and $0\leq \mu \leq 1$  satisfy  $\epsilon \lambda_{p_*}(1+ \delta^{\mu}K_{p_*}) <1$. Then, there exists   $0 < {\hat \delta} < \delta$ such that   for each  $p_{0}\in B_{\hat \delta}(p_{*})\backslash \{{p_*}\}$,    $\{p_{k}\}$ in Algorithm~\ref{Alg:NNM} is well defined, belongs to  $B_{\hat \delta}(p_{*})$, and converges to $p_{*}$ with order $1+\mu$ as follows:
\begin{equation*} \label{eq;csls}
d\left(p_{k+1},p_{*}\right)\leq \frac{ \epsilon\lambda_{p_*}K_{p_*}}{1 - \epsilon \lambda_{p_*}}  d(p_{k},p_{*})^{1+\mu},  \qquad k=0, 1, \ldots.
\end{equation*}
\end{theorem}

\section{Global Version of the Newton Method } \label{sec:nmca}
In this section, we propose e analyze a global version of the Newton Method (GNM) for finding a singularity of locally Lipschitz continuous vector fields on $\mathbb{M}$, i.e., to solve problem~\eqref{eq:TheProblem}. Basically, the GNM consists of combining the Newton method given by Algorithm~\ref{Alg:NNM}, with the nonmonotone line search technique of \cite{Grippo1986}. In particular, this technique guarantees a nonmonotone decrease of the merit function  $\varphi: \mathbb{M} \rightarrow \mathbb{R}$ defined by
\begin{equation}\label{eq:FM}
\varphi(p) :=\dfrac{1}{2}\|X(p)\|^2,
\end{equation}
with $\|\cdot\|$ denoting the Euclidean norm. To analyze the global convergence of the proposed scheme, we assume throughout this paper that the function $\varphi$ is continuously differentiable, even though $X$ itself is not. Moreover, the gradient de $\varphi$ at $p$  is explicitly computable using any element of the Clarke--generalized covariant derivative of $X$, i.e., $\varphi'(p) = V^TX(p)$ for any $V \in \partial X(p)$. In the following, we formally state the GNM to solve problem~\eqref{eq:TheProblem}.
\\
\hrule
\begin{algorithm}  \label{Alg:DNNM}
{\vspace{0.2cm}\bf Global Version of the Newton Method \vspace{0.3cm}}
\hrule
\vspace{0.1cm}
\begin{description}
\item[\bf Step 0.] Choose parameters $\beta\in (0,1)$ and $\sigma \in (0, 1/2)$. Let $p_0\in \mathbb{M}$ and $M \geq 0$ be given. Set $k=0$ and $m_0 = 0$.
\item[\bf Step 1.] If $X(p_k) = 0$, then \textbf{stop}.
\item[\bf Step 2.] Choose a $V_k := V_{p_k}\in \partial X(p_k)$ and compute $v_k := v_{p_k}\in T_{p_k}\mathbb{M}$ as a solution of the linear equation
\begin{equation} \label{eq:SDNNM}
X(p_k) + V_{k} v = 0.
\end{equation}
If such $v_k$ exists go to \textbf{Step 3}; otherwise, set $v_k = - \varphi'(p_k) = -[V_k^T X(p_k)]$, with $\varphi$ defined by \eqref{eq:FM}.
\item[\bf Step 3.] If $v_k = 0$, then \textbf{stop}. Otherwise, set $\alpha = 1$ and do $\alpha = \beta \alpha$, while
\begin{equation}\label{ine:armijo}
\varphi(\exp_{p_k}\left(\alpha v_k)\right) > \max_{0\leq j \leq m_k}\{\varphi(p_{k-j})\} + \sigma \alpha \varphi'(p_k)^Tv_k,
\end{equation}
with the nonmonotone index function $m_k \leq \min\{m_{k-1} + 1, ~ M\}$ for all $k \geq 1$.
\item[\bf Step 4.] Set $\alpha_k = \alpha$, update $p_{k+1}:= \exp_{p_k}(\alpha v_k)$, $k \leftarrow k+1$, and go to \textbf{Step~1}.
\end{description}
\hrule
\end{algorithm}
\vspace{0.3cm}

\begin{remark}
Notably, to guarantee the well-definedness of a sequence generated by Algorithm~\ref{Alg:DNNM}, we should to check in each iteration $k$ three issues: $(i)$ the Clarke--generalized covariant derivative $\partial X(p_k)$ must be nonempty, see \cite[Proposition 3.1]{deOliveiraFerreira2020}; $(ii)$ all element $V_k \in \partial X(p_k)$ must be non-singular, see Lemma~\ref{le:NonSing}; and $(iii)$ the search direction $v_k\in T_{p_k}\mathbb{M}$ obtained in Setp 2 must be a descent direction for $\varphi$ at $p_k$. The last condition is discussed in the following. Finally, we remark that if $X$ is continuously differentiable and $M = 0$ our method is equivalent to the method proposed in \cite{FernandesFerreiraYuan2020}.
\end{remark}

In the following, we present an useful result for establishing the well-definedness of a sequence generated by Algorithm~\ref{Alg:DNNM}. Its proof is similar to \cite[Lemma 3]{FernandesFerreiraYuan2020} and, we decided to present the proof here for the sake of completeness.

\begin{lemma}\label{le:Welldefined1}
Suppose that $p \in \mathbb{M}$ is such that $X(p) \neq 0$ and $V \in \partial X(p)$. Assume that $v = -V^TX(p)$ or that  is a solution of the following linear equation
\begin{equation}\label{eq:ELNM}
X(p) + Vv = 0.
\end{equation}
If $v \neq 0$, then $v$ is a descent direction for $\varphi$ at $p$.
\end{lemma}
\begin{proof}
Firstly, assume that $v = -V^TX(p)$. Since $\varphi'(p) = V^TX(p)$, we have $\varphi'(p)^{T}v = -\|V^T X(p)\|^2 < 0$. Now, suppose that $v$ is a solution of \eqref{eq:ELNM}, i.e., $v = - V^{-1}X(p)$. Again using the fact that $\varphi' (p) = V^{T} X(p)$, we obtain that $\varphi'(p)^Tv = X(p)^{T}Vv$. By the property of the norm, we have $\varphi'(p)^Tv = -\|X(p)\|^2 < 0$, since $X(p)\neq 0$. Therefore, for both choices, we conclude that $v$ is a descent direction for $\varphi$ at $p$.
\end{proof}

\subsection{Global Convergence Analysis}
In this section, we shall present and prove a result on the global convergence of the GNM. We show that under natural assumptions, this method is well defined and preserves the fast convergence rates of the Newton method described by Algorithm~\ref{Alg:NNM}. We begin by showing that the GNM is well defined, i.e., the Setp 3 in Algorithm~\ref{Alg:DNNM} is satisfied in a finite number of backtrackings.

\begin{lemma}
Let $\{p_k\}$ be a sequence generated using Algorithm~\ref{Alg:DNNM}. Then $\{p_k\}$ is well defined.
\end{lemma}
\begin{proof}
Let $p_0 \in \mathbb{M}$ and suppose that $X(p_0)\neq 0$. According to Lemma~\ref{le:Welldefined1}, we obtain that $v_0 = - V^{-1}_{0}X(p_0)$ or $v_0 = V_0^TX(p_0)$ are such that $\varphi'(p_0)^Tv_0 < 0$. Since $\varphi: \mathbb{M} \rightarrow \mathbb{R}$ is a continuously differentiable function and $\sigma \in (0,1/2)$, we have
$$
\lim_{t \downarrow 0}\dfrac{\varphi(\exp_{p_0}(tv_0)) - \varphi(p_0)}{t} = \varphi'(p_0)^Tv_0 \leq \sigma \varphi'(p_0)^Tv_0  < 0.
$$
Therefore, it is straightforward to show that there exists $\delta \in (0,1]$ such that
$$
\varphi(\exp_{p_0}(tv_0)) < \varphi(p_0) + \sigma t \varphi'(p_0)v_0 = \varphi(p_{l(0)})+ \sigma t \varphi'(p_0)v_0, \qquad  t \in (0,\delta)
$$
The last inequality implies that $\alpha_0$ is well defined. Hence, $p_1$ generated using Algorithm~\ref{Alg:DNNM} is well defined. Using an induction argument, we can prove that $\{p_k\}$ is well define and the proof of lemma is concluded.
\end{proof}

In the next theorem, we will also show that all limit points of the sequence generated by the Algorithm~\ref{Alg:DNNM} are singularities for the vector field $X$. To this end, we assume that the sequence generated using Algorithm~\ref{Alg:DNNM} is infinite, and that $v_k \neq 0$ and $X(p_k)\neq0$ for all $k = 0, 1, \ldots$. Otherwise, if $\{p_k\}$ is finite, then the last iterate is a solution of problem~\eqref{eq:TheProblem} or a stationary point of the merit function $\varphi$ defined in \eqref{eq:FM}.

\begin{theorem}\label{th:mainss}
Let  $X$ be a locally Lipschitz continuous vector field on $\mathbb{M}$. Assume that $X$ is regular at ${p_*} \in \mathbb{M}$, the level set $\Omega_0 := \{p \in \mathbb{M}:~\varphi(p)\leq \varphi(p_0)\}$ is bounded, and $\{p_k\}$ generated using Algorithm~\ref{Alg:DNNM} has a accumulation point ${p_*}$. If $\{v_k\}$ is bounded, then $p_*$ is a singularity of $X$.
\end{theorem}
\begin{proof}
We show that $p_* \in \mathbb{M}$ is such that $X(p_*) = 0$ whenever $\{v_k\}$ is bounded, by adapting the proof present in \cite[Theorem 1]{Grippo1986}. Without loss of generality, we assume that $\varphi'(p_k) \neq0$ and $X(p_k)\neq 0$ for all $k = 0,1, \ldots$. Owing to $X$ be regular at $p_*$ by using Lemma~\ref{le:NonSing} for every $\epsilon > 0$ satisfying $\epsilon \lambda_{p_*} < 1$, there exists $0 < \delta < r_{p_*}$ such that $V_k \in \partial X(p_k)$ is non-singular for all $p_k \in B_\delta(p_*)$. Hence, \eqref{eq:SDNNM} has a solution and $v_k = -V_k^{-1}X(p_{k})$ for $k = 0,1, \ldots$. Let $l(k)$ be an integer such that $k - m_k \leq l(k) \leq k$ and
\begin{equation}\label{eq:FIM}
\varphi(p_{l(k)}) := \max_{0\leq j\leq m_k}\varphi(p_{k-j}).
\end{equation}
Since $m_{k+1}\leq m_k + 1$, it follows that $\{\varphi(p_{l(k)})\}$ is monotonically nonincreasing, and from the boundless of $\Omega_0$, we ensure that $\{\varphi(p_{l(k)})\}$ has a limit as $k$ goes to infinity. From \eqref{ine:armijo} and \eqref{eq:FIM}, for $k > M$, we have
\begin{equation}\label{eq:FIC}
\varphi(p_{l(k)}) \leq \varphi(p_{l(k) - 1}) + \sigma \alpha_{l(k) - 1}\varphi'(p_{l(k) - 1})^Tv_{l(k) - 1}.
\end{equation}
Because $\alpha_{l(k) - 1} > 0$ and $\varphi'(p_{l(k) - 1})^Tv_{{l(k) - 1}} < 0$, by taking limits  as $k$ goes to infinity, in \eqref{eq:FIC}, it follows that $\lim_{k\rightarrow \infty}[\alpha_{l(k) - 1}\varphi'(p_{l(k) - 1})^Tv_{l(k) - 1}] = 0$ and following the idea in the proof of \cite[Theorem 1 (a)]{Grippo1986}, we can write $\lim_{k\rightarrow \infty}[\alpha_{k}\varphi'(p_k)^Tv_k] = 0$. Now, let $p_*$ be an accumulation point of $\{p_k\}$, and relabel $\{p_k\}$ a subsequence converging to $p_*$. Hence, there exists a subsequence of indices $K\subset \mathbb{N}$ such that
\begin{equation}\label{eq:LIM}
\lim_{k \in K}[\alpha_{k}\varphi'(p_{k})^Tv_{k}] = 0.
\end{equation}
We have two possible cases to consider: $\limsup_{k \in K} \alpha_{k} > 0$ or $\lim_{k \in K} \alpha_{k} = 0$. In the first case, passing onto a further subsequence if necessary, we can assume from \eqref{eq:LIM} that $\lim_{k \in K_1}\varphi'(p_{k})^Tv_{k} = 0$ where $K_1 \subset K$. Because $\{v_k\}$  is bounded, there exists a subsequence of indices $K_2\subset K_1$ such that $\lim_{k \in K_2}v_{k} = v_* \neq 0$. Moreover, using that $\lim_{k \in K_2}p_{k} = p_*$ and that $\varphi$ is continually differentiable, we obtain that $\varphi'(p_*)^Tv_* = 0$. In addition, since $v_{k} = -V_{k}^{-1}X(p_{k})$ and $\varphi'(p_{k})^Tv_{k} = - \|X(p_{k})\|^2$, we conclude that $X(p_*) = 0$ what means $p_*$ is a singularity of $X$. Now, we assume that second case holds, i.e., $\lim_{k \in K} \alpha_{k} = 0$. Let $\alpha_{k}$ be chosen in the Step 3 of Algorithm~\ref{Alg:DNNM} such that $\alpha_{k} = \bar{\alpha}_{k}/2$, where $\bar{\alpha}_{k}$ was the last step that satisfy \eqref{ine:armijo}, i.e.,
\begin{equation}\label{eq:Non}
\varphi(\exp_{p_{k}}(\bar{\alpha}_{k}v_{k})) > \max_{0\leq j\leq m_k}\{\varphi(p_{k - j})\} + \sigma \bar{\alpha}_{k}\varphi'(p_{k})^Tv_{k} \geq \varphi(p_{k}) + \sigma \bar{\alpha}_{k} \varphi'(p_{k})^Tv_{k}, \qquad k \in K.
\end{equation}
Using equality~\eqref{eq:DefGra}, the fact that $\exp_p(0) = p$ for $p\in \mathbb{M}$, and the mean value theorem applied to the smooth function $\varphi \circ \exp_{p_k}: T_{p_k}\mathbb{M} \rightarrow \mathbb{R}$, there exists $t_{k}\in [0, \bar{\alpha}_k]$ such that
\begin{equation}\label{eq:TheVM}
\dfrac{\varphi(\exp_{p_{k}}(\bar{\alpha}_{k}v_{k})) - \varphi(p_{k})}{\bar{\alpha}_{k}} = D(\varphi \circ \exp_{p_k})(t_kv_k)[v_k] =  \varphi'(\exp_{p_k}(t_kv_k))^T[D\exp_{p_k}(t_kv_k)[v_k]],
\end{equation}
for $k \in K$. Combining \eqref{eq:Non} and \eqref{eq:TheVM}, we obtain that $\varphi'(\exp_{p_k}(t_kv_k))^T[D\exp_{p_k}(t_kv_k)[v_k]] > \sigma \varphi'(p_k)^Tv_k$. Since $\{v_k\}$ is bounded, there exists a subsequence of indices $K_1 \subset K$ such that $\lim_{k\in K_1}v_k = v_* \neq 0$, and thus $\lim_{k\in K_1}\{t_kv_k\} = 0$. Using that the exponential mapping is smooth and $\lim_{k\in K_1}p_{k} = p_*$, we have $\lim_{k \in K_1}\{\exp_{p_k}(t_kv_k)\} = \exp_{p_*}(0) = p_*$ and $\lim_{k\in K_1}D\exp_{p_k}(t_kv_k)[v_k] = D\exp_{p_*}(0)[v_*] = v_*$. Hence, owing to $\varphi$ be continuous differentiable taking limit on the last inequality, we get
$$
\lim_{k\in K_1}\varphi'(\exp_{p_k}(t_kv_k))^T[D\exp_{p_k}(t_kv_k)[v_k]] = \varphi'(p_*)^T v_* \geq \sigma \varphi'(p_*)^Tv_*.
$$
This implies that $(1- \sigma)\varphi'(p_*)^Tv_* \geq 0$, which may only holds when $\varphi'(p_*)^Tv_* \geq 0$ since $1 - \sigma > 0$. On the other hand, since $\varphi'(p_{k})^Tv_{k} < 0$, by taking limit, we conclude that $\varphi'(p_*)^Tv_* \leq 0$, which combined with  $\varphi'(p_*)^Tv_* \geq 0$, yields  $\varphi'(p_*)^Tv_* = 0$. So, since $\varphi'(p_*)^Tv_* = - \|X(p_*)\|^2$ and $v_*$ is a descent direction for $\varphi$ at $p_*$, we conclude that $X(p_*) = 0$, and proof theorem is complete.
\end{proof}

Now, we prove a global convergence theorem for the GNM. In particular, we show that after a finite number of iteration, our method reduces to the NM described by Algorithm~\ref{Alg:NNM}. Consequently, under natural assumptions the fast local convergence for the GNM is preserved.
\begin{theorem}
Let $X$ be a locally Lipschitz continuous vector field on $\mathbb{M}$. Assume that $X$ satisfies \textbf{A1} with ${\bar{p} = p_*}$, the level set $\Omega_0 := \{p \in \mathbb{M}:~\varphi(p)\leq \varphi(p_0)\}$ is bounded and $\{p_k\}$ generated using Algorithm~\ref{Alg:DNNM} has a accumulation point ${p_*}$. Then, $p_*$ is a singularity of $X$ and $\{p_k\}$  generated by Algorithm~\ref{Alg:DNNM} with $\sigma \in (0, 1/2)$ converges to $p_{*}$ with order $1 + \mu$.
\end{theorem}
\begin{proof}
Since $X$ satisfies \textbf{A1} with ${\bar{p} = p_*}$, we can take $\lambda_{p_*}>0$ such that $\lambda_{p_*} \geq \max\{\|V_{p_*}^{-1}\|: ~ V_{p_*}\in \partial X(p_*)\}$. So, it follows from Lemma~\ref{le:NonSing} that for every $\epsilon > 0$, $0 \leq \mu \leq 1$ and $0 < \delta < r_{p_*}$  satisfying $\epsilon \lambda_{p_*}(1+\delta^\mu K_{p_*}) < 1$, we have
\begin{equation}\label{eq:LIM1}
\left\|V_{k}^{-1}\right\|\leq \dfrac{\lambda_{p_*}}{1 - \epsilon\lambda_{p_*}}, \qquad \forall ~ p_k \in B_{\delta}(p_*), \quad \forall ~V_{k}\in \partial X(p_k).
\end{equation}
On the other hand, using that the parallel transport is an isometry and some algebraic manipulations, we obtain that $\|X(p_k)\| \leq \|X(p_*) - P_{p_kp_*}[X(p_k) + V_{k}\exp_{p_k}^{-1}p_*]\| + \|X(p_*)\| + \|V_{k}\|\|\exp_{p_k}^{-1}p_*\|$. Now, since $X$ satisfies \eqref{eq:scA1}, all element $V_{k} \in \partial X(p_k)$ is such that $\|V_{k}\| \leq L$, where $L>0$ representees the Lipschitz constant of $X$ around $p_*$, see \cite[Proposition 3.1 (ii)]{deOliveiraFerreira2020}. Moreover, using that $\|\exp_{p_k}^{-1}p_*\|= d(p_k, p_*) < \delta$ the last inequality reduces to
\begin{equation}\label{eq:LIM2}
\big\|X(p_k)\big\|  <  (\epsilon + L)\delta^\mu + \big\|X(p_*)\big\|.
\end{equation}
This implies that $\{X(p_k)\}$ is bounded. Combining \eqref{eq:SDNNM}, \eqref{eq:LIM1} and \eqref{eq:LIM2}, we have $\{v_k\}$ is bounded, and according to Theorem~\ref{th:mainss}, we conclude that $X(p_*) = 0$. Now, we turn to the convergence rate. To this end, we proceed to prove that there exists an integer $k_0 > 0$ such that for all $k\geq k_0$, we have $\alpha_k = 1$; hence $p_{k+1} = \exp_{p_k}(v_k)$, where $v_k = -V_k^{-1}X(p_k)$. Indeed, since $V_{p_*}$ is non-singular and $X(p_*) = 0$, Lemma~\ref{le:welldefined} implies that there exists $0<\hat{\delta} < \delta$ such that $V_{k} \in \partial X(p_k)$ is non-singular for all $p_k \in B_{\hat{\delta}}(p_*)$ and $N_X(p_k)\subset B_{\hat{\delta}}(p_*)$. Consequently, $p_{k+1}\in B_{\hat{\delta}}(p_*)$. Since $p_*$ is a cluster point of $\{p_k\}$, there exists $k_0>0$  such that shrinking $\hat{\delta}$ if necessary, we have $p_{k_0} \in B_{\hat{\delta}}(p_*)$. Let $\tilde{p}_{k_0} := \exp_{p_{k_0}}(v_0)$. Considering that $X(p_*) = 0$ and the parallel transport is an isometry,  we obtain by \eqref{eq:scA1} with $\mu = 0$ that
\begin{equation}\label{eq:Eq2}
\left\| X(\tilde{p}_{k_0})\right\| - \left\|V_{k_0 + 1} \exp_{\tilde{p}_{k_0}}^{-1}p_{*}\right\| \leq \epsilon d(\tilde{p}_{k_0}, p_*).
\end{equation}
Since $\tilde{p}_{k_0}\in B_{\hat{\delta}}(p_*)$, we can apply \cite[Proposition 3.1 (ii)]{deOliveiraFerreira2020} to conclude that all element $V_{k_0+1}\in \partial X(\tilde{p}_{k_0})$ is such that $\|V_{k_0+1}\| \leq L$. Moreover, since $\|\exp_{\tilde{p}_{k_0}}^{-1}p_*\| = d(\tilde{p}_{k_0}, p_*)$, inequality~\eqref{eq:Eq2} reduces to $\|X(\tilde{p}_{k_0})\| \leq (\epsilon + L) d(\tilde{p}_{k_0}, p_*)$. Now, using Lemma~\ref{le:welldefined} with $p = p_{k_0}$ and $\mu = 0$, we obtain that
\begin{equation}\label{eq:Eq3}
\left\|X(\tilde{p}_{k_0})\right\| \leq (\epsilon + L) \dfrac{\epsilon \lambda_{p_*}K_{p_*}}{1 - \epsilon \lambda_{p_*}}d(p_{k_0}, p_*).
\end{equation}
Because $\tilde{p}_{k_0} = \exp_{p_{k_0}}(-V_0^{-1}X(p_{k_0}))$ and $\|V_0^{-1}\|\leq \lambda_{p_*}/(1-\epsilon\lambda_{p_*})$ using triangular inequality, the definition of the exponential mapping and Lemma~\ref{le:welldefined}, we have
$$
d(p_{k_0},p_*) \leq \dfrac{\lambda_{p_*}}{1 - \epsilon \lambda_{p_*}}\left\|X(p_{k_0})\right\| + d(\tilde{p}_{k_0}, p_*) \leq \dfrac{\lambda_{p_*}}{1 - \epsilon \lambda_{p_*}}\left\|X(p_{k_0})\right\| + \dfrac{\epsilon \lambda_{p_*} K_{p_*}}{1-\epsilon \lambda_{p_*}}d(p_{k_0}, p_*).
$$
This implies that $d(p_{k_0}, p_*) \leq [\lambda_{p_*}/(1-\epsilon \lambda_{p_*}(1 + K_{p_*}))]\|X(p_{k_0})\|$. Thus, combining this inequality  with \eqref{eq:Eq3} we conclude that $\|X(\tilde{p}_{k_0})\| \bar{\epsilon} \leq \|X(p_{k_0})\|$, where $\bar{\epsilon} = (\epsilon + L)\epsilon \lambda_{p_*}^2K_{p_*}/(1 - \epsilon \lambda_{p_*})[1-\epsilon \lambda_{p_*}(1 + K_{p_*})]$. By \eqref{eq:FM}, since $\varphi'(p_{k_0})^Tv_{0} = -\|X(p_{k_0})\|^2$, we can take $\epsilon > 0$ such that  $\bar{\epsilon} \leq \sqrt{1 - 2\sigma}$ to conclude that
$$
\varphi(\tilde{p}_{k_0}) = \dfrac{1}{2}\left\|X(\tilde{p}_{k_0})\right\|^2 \leq \dfrac{1 - 2\sigma}{2}\left\|X(p_{k_0})\right\|^2  \leq \max_{0\leq j \leq m_{k_0}}\{\varphi(p_{k_0 - j})\} + \sigma \varphi'(p_{k_0})^Tv_{0}.
$$
From \eqref{ine:armijo}, we have $\alpha_{k_0} = 1$; hence $p_{k_0 + 1} = \tilde{p}_{k_0}$. This implies that $p_{k_0 + 1} \in B_{\hat{\delta}}(p_*)$ since $N_X(p_{k_0}) \subset B_{\hat{\delta}}(p_*)$. By induction of the above arguments, we obtain that
\begin{equation}\label{eq:NEARM}
\alpha_k = 1, \quad p_{k+1} \in N_X(p_k) \subset B_{\hat{\delta}}(p_*), \quad \forall ~k\geq k_0.
\end{equation}
Since $\epsilon > 0$ and $0 < \delta < r_{p_*}$ satisfy $\epsilon \lambda_{p_*}(1 + \delta^\mu K_{p_*}) < 1$, we can apply Theorem~\ref{th:conv} to conclude from \eqref{eq:NEARM} that the sequence $\{p_k\}$ generated by Algorithm~\ref{Alg:DNNM} converges with order $1+\mu$ to $p_*$.
\end{proof}

\section{Numerical Experiments}\label{sec:numexp}
This section reports some preliminary numerical experiments obtained by applying the GNM and NM under a class of locally Lipschitz continuous vector fields. Before setting the problem, we begin by presenting some preliminaries on the sphere geometry. Let $\langle \cdot, \cdot\rangle$ be the {\it usual inner product on $\mathbb{R}^{n}$}, with corresponding {\it norm} denoted by $\| \cdot\|$. The {\it $(n-1)$--dimensional Euclidean sphere} and its {\it tangent hyperplane at a point $p$} are denoted, respectively, by
$$
\mathbb{S}^{n-1}:=\left\{ p=(p_1, \ldots, p_{n}) \in \mathbb{R}^{n}: ~ \|p\|= 1\right\}, \qquad
T_{p}\mathbb{S}^{n-1}:=\left\{v\in \mathbb{R}^{n}:~ \langle p, v \rangle=0 \right\}.
$$
Now, we present the problem at stake. The \textit{absolute value vector field} (AVVF) is described as:
$$
\mbox{find} \quad p\in\mathbb{S}^{n-1}\quad \mbox{such that} \quad (I-pp^T)\left[Ap - |p| - b\right] = 0,
$$
where $I$ denote the $n\times n$ identity matrix, $I-pp^T: \mathbb{R}^{n} \to T_p\mathbb{S}^{n-1}$ is the linear mapping denominated by {\it projection onto the tangent hyperplane} $T_p\mathbb{S}^{n-1}$, $A \in \mathbb{R}^{n\times n}$, $b \in \mathbb{R}^n \equiv \mathbb{R}^{n\times 1}$, and $|p|$ denotes the vector whose $i$-th component is equal to $|p_i|$. In our implementation, each AVVF was generated randomly. We used the Matlab routine \textit{sprand} to construct matrix $A$. In particular, this routine generates a sparse matrix with a predefined dimension, density, and singular values. Initially, we defined the dimension $n$ of the problem and the density. Next, we randomly generated the vector of singular values from a uniform distribution on $(0, 1)$. To ensure that the condition $\|A^{-1}\|< 1/3$ is fulfilled and consequently the well-definedness of the method be guaranteed (see \cite[Theorem 2]{BelloCruz2016}), we rescale the vector of singular values. To generate the vector $b$, we selected a random solution $p_* \in \mathbb{S}^{n-1}$ from a uniform distribution on $(-100, 100)$ and computed $b = Ap_* - |p_*|$. In both methods, we choose $p_0 \in \mathbb{S}^{n-1}$ uniform distribution on $(-100, 100)$ as the starting point. We stopped the execution of Algorithm~\ref{Alg:DNNM} at $p_k$, declaring convergence if $\|(I-p_kp_k^T)[Ap_k - |p_k| - b]\| < 10^{-6}$. In case this stopping criterion is not satisfied, the method stops when a maximum of $100$ iterations has been performed. For this class of problems, an element of the Clarke--generalized Jacobian at $p$ (see \cite{deOliveiraFerreira2020,BelloCruz2016,Mangasarian2009}) is given by
$$
V = (I - p p^T)[A -  \mbox{diag}\big(\mbox{sgn}(p)\big)] - p^T[Ap - |p| - b]I, \qquad p \in  \mathbb{S}^{n-1},
$$
where $\mbox{diag}(\alpha)$ denotes a diagonal matrix with diagonal elements $\alpha_1, \alpha_2, \ldots, \alpha_n$, and $\mbox{sgn}(p)$ denotes a vector with components equal to $-1$, $0$, or $1$ depending on whether the corresponding component of the vector $p$ is negative, zero, or positive, respectively. The numerical results were obtained using Matlab version R2016a on a 2.5~GHz Intel\textregistered\ Core\texttrademark\ i5 2450M computer with 6~GB of RAM and Windows 7 ultimate system.

\begin{table}[h!]
\setstretch{1.5}
\centering
\caption{Performance of the GNM and NM}\label{tab1}
\vspace{0.5cm}
\begin{tabular}{c|c|c|c|c|c|c|c|c|c|c|c|c|}
\cline{2-13}
& \multicolumn{3}{c|}{GNM (M = 0)}   &   \multicolumn{3}{c|}{GNM (M = 1)}   &   \multicolumn{3}{c|}{GNM (M = 5)}   &   \multicolumn{3}{c|}{NM}\\ \hline
\multicolumn{1}{|c|}{Dimension} & \% & It & Time      & \% & It &  Time      & \% &  It  &  Time       & \% & It &   Time     \\ \hline
\multicolumn{1}{|c|}{100}       & 99 & 40 & 13.96     & 96 & 29 &  13.58     & 98 &  32  &  17.58      & 95 & 20 &  18.74     \\ \hline
\multicolumn{1}{|c|}{400}       & 97 & 28 & 14.96     & 92 & 25 &  14.53     & 96 &  36  &  20.61      & 86 & 20 &  18.69     \\ \hline
\multicolumn{1}{|c|}{800}       & 97 & 25 & 15.36     & 100& 23 &  14.82     & 96 &  23  &  23.38      & 94 & 24 &  22.20     \\ \hline
\multicolumn{1}{|c|}{1600}      & 98 & 34 & 17.14     & 96 & 31 &  16.91     & 90 &  34  &  23.43      & 95 & 25 &  23.80     \\ \hline
\end{tabular}

\end{table}

Table~\ref{tab1} display the numerical results obtained for AVVFs of dimensions $100$, $400$, $800$, and $1600$. For the numerical tests, we considered three options distinct for the constant $M$, namely $M = 0$, $M = 1$, and $M = 5$. It is worth pointing out that for $M = 0$ the equality~\eqref{ine:armijo} reduces to Armijo condition for all $k = 0,1, \ldots$. The methods were compared on the percentage of problems solved $(\%)$, average number of iterations (It), and average CPU time in seconds (Time). We generated 100 AVVFs of dimensions 100, 400, 800, and 1600. The density of matrix $A$ was set to $0.003$, similar to that in [6]. This implies that only approximately $0.3\%$ of the elements of $A$ are non-null. We executed each test problem for three times and defined the corresponding CPU time as the mean of these measurements to reach a higher accuracy of the CPU time.

Analyzing Table~\ref{tab1}, we can observe that for the set of test problems, the strategy of globalization becomes the NM more robust. For example, for AVVEs of dimension 800, the robustness rate of NM was $94\%$, while that of GNM ($M = 1$) was $100\%$. Regarding the average number of iterations, we observe that in most cases the NM is better than GNM ($M = 0$, $M = 1$ and $M = 5$). However, if we analyze the average time spent in solving the problems, it is possible to note that the behavior of both versions of GNM ($M = 0$ and $M = 1$) is better than of NM.

In summary, these experiments indicate that the constant $M$ interferes with the behavior of GNM. For example, the Armijo condition $(M = 0)$ was shown to be inferior to the nonmonotone technique $(M = 1)$ both in terms of the number of interactions and the average time.

\section{Conclusions}\label{sec:fr}
This paper proposed and analyzed the GNM for finding a singularity of the nonsmooth vector fields. It basically consists of combining our first algorithm, see \cite{deOliveiraFerreira2020}, with a nonmonotone line search technique. Under suitable conditions, global convergence of the algorithm to a stationary point of the chosen merit function was established. Some numerical experiments were carried out in order to illustrate the numerical behavior of the methods. They indicate that the proposed schemes represent a promising alternative for finding a singularity of the nonsmooth vector fields.

\end{document}